\documentclass[11pt]{article}
\usepackage{amssymb,amsmath,amsfonts,amsthm}

\usepackage{epsfig}

\numberwithin{equation}{section}

\newcommand{\di}{\displaystyle}

\newcommand{\R}{{\mathbb R}}
\newcommand{\intr}{\int\limits_{\R^N}}
\newcommand{\ia}{I_\alpha}
\newcommand{\mj}{{\mathcal J}}
\newcommand{\me}{{\mathcal E}}

\theoremstyle{plain}

\newtheorem{proposition}{Proposition}[section]
\newtheorem{cor}{Corollary}

\newtheorem{theorem}{Theorem}[section]
\newtheorem{lemma}{Lemma}[section]

\theoremstyle{definition}

\begin{document}

\title{\vskip-0.3in On a class of mixed Choquard-Schr\"odinger-Poisson systems}

\author{Marius Ghergu\footnote{School of Mathematics and Statistics,
    University College Dublin, Belfield, Dublin 4, Ireland; {\tt
      marius.ghergu@ucd.ie}}
      $\;\;$       and $\;$
Gurpreet Singh\footnote{School of Mathematics and Statistics,
    University College Dublin, Belfield, Dublin 4, Ireland; {\tt gurpreet.singh@ucdconnect.ie}} 
}

\maketitle

\begin{abstract}
We study the system
$$
\left\{
\begin{aligned}
-\Delta u+u+K(x) \phi |u|^{q-2}u&=(I_\alpha*|u|^p)|u|^{p-2}u &&\mbox{ in }\R^N,\\
-\Delta \phi&=K(x)|u|^q&&\mbox{ in }\R^N,
\end{aligned}
\right.
$$
where $N\geq 3$, $\alpha\in (0,N)$, $p,q>1$ and $K\geq 0$. Using a Pohozaev type identity we first derive conditions in terms of $p,q,N,\alpha$ and $K$ for which no solutions exist. Next, we discuss the existence of a ground state solution by using a variational approach.

\medskip


\end{abstract}


\section{Introduction}\label{sec1}
In this paper we are concerned with the following system 
\begin{equation}\label{sp}
\left\{
\begin{aligned}
-\Delta u+u+K(x) \phi |u|^{q-2}u&=(I_\alpha*|u|^p)|u|^{p-2}u &&\mbox{ in }\R^N,\\
-\Delta \phi&=K(x)|u|^q&&\mbox{ in }\R^N,
\end{aligned}
\right.
\end{equation}
where $p,q>1$ are real numbers and $K\geq 0$ satisfies some more properties as we shall state below. Here $I_\alpha:\R^N\to \R$ is the {\it Riesz potential} of order $\alpha\in (0,N)$, $N\geq 3$,  given by
\begin{equation}\label{rieszdef}
I_\alpha(x)=\frac{A_\alpha}{|x|^{N-\alpha}}\,,\quad \mbox{ with }\; A_\alpha=\frac{\Gamma\big(\frac{N-\alpha}{2}\big)}{\Gamma(\frac{\alpha}{2}\big) \pi^{N/2}2^\alpha}.
\end{equation}

When $K\equiv 0$, system \eqref{sp} reduces to the single equation
\begin{equation}\label{gstate}
-\Delta u+u=(I_\alpha*|u|^p)|u|^{p-2}u\quad \mbox{ in }\R^N
\end{equation}
which bears the name {\it Choquard} or {\it Choquard-Pekar equation}.

For $N=3$, $p=\alpha=2$, equation \eqref{gstate} was introduced in 1954  by S.I. Pekar \cite{P1954} as a
model in quantum theory of a Polaron at rest (see also
\cite{DA2010}). In 1976, P. Choquard  used \eqref{gstate} in a certain approximation to Hartree-Fock theory of one component plasma (see \cite{L1976}). In 1996, equation \eqref{gstate} appears in a different context, being employed by R. Penrose \cite{P1996} as a model of
self-gravitating matter (see, e.g., \cite{J1995,MPT1998}) and it is
known in this context as the {\it Schr\"odinger-Newton equation}.

If $u$ solves \eqref{gstate}, then the function $\psi$ defined by $\psi(t,x)=e^{it}u(x)$ is a solitary wave of the focussing time dependent Hartree-Fock equation
$$
i\psi_t+\Delta\psi+(\ia*|\psi|^p)|\psi|^{p-2}\psi=0\quad\mbox{ in }\R_+\times \R^N.
$$
The Choquard equation \eqref{gstate} has been investigated
for a few decades by variational methods starting with the pioneering
works of E.H. Lieb \cite{L1976} and P.-L. Lions \cite{Lions1980,Lions1984}. More
recently, new and improved techniques have been devised to deal with
various forms of \eqref{gstate} (see, e.g.,
\cite{AFM2016,MV2013,MV2015a,MV2015b,MV2015c,WW2009} and the references
therein). In \cite{MV2013} existence, regularity, positivity, asymptotic behavior and radial symmetry of solutions to \eqref{sp} is discussed for optimal range of parameters. We also mention here the works \cite{CL2016,DSS2015} where the fractional version of \eqref{gstate} is considered. For a nonvariational approach to Choquard equation the reader may consult \cite{GT2016,MZ2010,MV2013b}.

Back to \eqref{sp}, we should point out that since for all $\varphi\in C^\infty_0(\R^N)$, $\ia*\varphi\to \varphi$ as $\alpha\to 0$, the system
 \begin{equation}\label{sp0}
\left\{
\begin{aligned}
-\Delta u+u+K(x) \phi |u|^{q-2}u&=|u|^{2p-2}u &&\mbox{ in }\R^N,\\
-\Delta \phi&=K(x)|u|^q&&\mbox{ in }\R^N,
\end{aligned}
\right.
\end{equation}
may be seen as a formal limit of \eqref{sp} when $\alpha\to 0$. 
The nonlocal nonlinear Schr\"odinger equation
$$
i\psi_t+\Delta\psi+V_{ext}(x)\psi+(I_2*|\psi|^2)\psi-|\psi|^{p-2}\psi=0\,, \quad (t,x)\in \R_+\times \R^3
$$
is used as an approximation to Hartree-Fock model of a quantum many-body system of electrons under the presence of an external potential $V_{ext}$ (see \cite{LL2005}). In such a setting, \eqref{sp0} and its stationary counterpart bear the name of Schr\"odinger-Poisson-Slater \cite{BLS2003}, Schr\"odinger-Poisson-$X_\alpha$ \cite{BMS2003,M2001}, or Maxwell-Schr\"odinger-Poisson \cite{BF1998,CDSO2013} equations. The convolution term in \eqref{sp0} represents the Coulombic repulsion between the electrons. The local term $|u|^{2p-2}u$ was introduced by Slater \cite{S1951} as a local approximation of the exchange potential in the Hartree-Fock model \cite{BLS2003,M2001}.
\medskip

\noindent{\bf Notations.}
Throughout in this paper we use the following notations.

\begin{itemize}
\item $H^1(\R^N)$ denotes the standard Sobolev space endowed with the usual norm
$$
\|u\|^2=\intr\big(|\nabla u|^2+|u|^2\big)dx.
$$
We shall denote by $\langle\cdot,\cdot\rangle$ the duality pairing between $H^1(\R^N)$ and its dual $H^{-1}(\R^N)$.
\item ${\mathcal D}^{1,2}(\R^N)$ is the Hilbert space 
$$
{\mathcal D}^{1,2}(\R^N)=\{u\in L^{2^*}(\R^N):\, |\nabla u|\in L^2(\R^N)\}
$$
endowed with the standard norm
$$
\|u\|_{{\mathcal D}^{1,2}}^2=\intr |\nabla u|^2 dx.
$$
and the associated scalar product 
$$
\big(u,v\big)_{{\mathcal D}^{1,2}}=\intr \nabla u\cdot \nabla v.
$$
\item $L^s(\R^N)$ is the usual Lebesgue space in $\R^N$ of order $s\in [1,\infty]$ whose norm will be denoted by $\|\cdot \|_s$.
\end{itemize}

\section{Main Results}
Our first result provides sufficient conditions for the nonexistence of solutions to \eqref{sp}. 
\begin{theorem}\label{nonex}
Assume $K\in C^1(\R^N)$, $K\geq 0$. If one of the following holds
\begin{enumerate}
\item[(i)] $x\cdot\nabla K(x)+\gamma K(x)\geq 0$ in $\R^N$ for some $\gamma\in (-\infty,\frac{N+2}{2})$ and
\begin{equation}\label{nex1}
p\geq\frac{N+\alpha}{N-2}\quad\mbox{ and }\quad \frac{N+\alpha}{p}\leq \frac{N+2-2\gamma}{q};
\end{equation}
\item[(ii)] $x\cdot\nabla K(x)+\gamma K(x)\leq 0$ in $\R^N$ for some $\gamma\in \R$ and
\begin{equation}\label{nex2}
p\leq\frac{N+\alpha}{N}\quad\mbox{ and }\quad \frac{N+\alpha}{p}\geq \frac{N+2-2\gamma}{q};
\end{equation}
\end{enumerate}
then, the only solution $(u,\phi)$  of \eqref{sp} that satisfies 
\begin{equation}\label{reg1}
u\in H^1(\R^N)\cap L^{\frac{2Np}{N+\alpha}}(\R^N)\,,\quad
\phi\in H^1(\R^N)
\end{equation} 
and 
\begin{equation}\label{reg2}
K(x) |u|^q\in L^{\frac{2N}{N+2}}(\R^N)\,,\quad |\nabla u|\in H^{1}_{loc}(\R^N)\cap L^{\frac{2Np}{N+\alpha}}_{loc}(\R^N)
\end{equation}
is $u\equiv \phi\equiv 0$.
\end{theorem}

By taking $K\equiv 0$, for suitable choice of $\gamma$ in \eqref{nex1} and \eqref{nex2} we obtain that if $p\geq\frac{N+\alpha}{N-2}$ or $p\leq\frac{N+\alpha}{N}$ then the only solution of  \eqref{gstate} which satisfies $u\in H^1(\R^N)\cap L^{\frac{2Np}{N+\alpha}}(\R^N)$  and $|\nabla u|\in H^{1}_{loc}(\R^N)\cap L^{\frac{2Np}{N+\alpha}}_{loc}(\R^N)$ is the trivial one. We thus recover the result in \cite[Theorem 2]{MV2013}.

By taking $\gamma=0$ in Theorem \ref{nonex} we obtain:
\begin{cor}\label{cor1}
Let $K\equiv const>0$. If one of the following conditions hold
$$
p\geq\frac{N+\alpha}{N-2}\quad\mbox{ and }\quad \frac{N+\alpha}{p}\leq \frac{N+2}{q};
$$
or
$$
p\leq\frac{N+\alpha}{N}\quad\mbox{ and }\quad \frac{N+\alpha}{p}>\frac{N+2}{q},
$$
then the only solution $(u,\phi)$  of \eqref{sp} satisfying \eqref{reg1}-\eqref{reg2} is the trivial one. 
\end{cor} 

\begin{cor}\label{cor2}
Let $K(x)=(1+|x|^2)^{-\gamma/2}$.
If $\gamma\in [0,\frac{N+2}{2})$ and \eqref{nex1} holds or $\gamma\leq 0$ and \eqref{nex2} holds then 
the only solution $(u,\phi)$  of \eqref{sp} satisfying \eqref{reg1}-\eqref{reg2} is the trivial one. 
\end{cor}

Let us now discuss the existence of a solution to \eqref{sp}. Crucial to our approach will be the Hardy-Littlewood-Sobolev inequality
\begin{equation}\label{hls1}
\intr |\ia*u|^{\frac{Ns}{N-\alpha s}}\leq C\Big(\intr |u|^s\Big)^{\frac{N}{N-\alpha s}}\quad \mbox{for any }u\in L^s(\R^N), s\in \big(1,\frac{N}{s}\big)
\end{equation}
which also implies
\begin{equation}\label{hls2}
\Big | \intr (\ia*u)v \Big|\leq C\|u\|_s\|v\|_t\quad \mbox{for any }u\in L^s(\R^N), v\in L^t(\R^N), \frac{1}{s}+\frac{1}{t}=1+\frac{\alpha}{N}.
\end{equation}

It is more convenient to reduce our system \eqref{sp} to a single equation. More exactly, for any $u\in H^1(\R^N)$ define
$$
T_u:{\mathcal D}^{1,2}(\R^N)\to \R\,,
\quad
T_u(v)=\intr K(x)|u|^q v dx.
$$
If $K\in L^r(\R^N)$, with
\begin{equation}\label{k}
\frac{1}{r}+\frac{q+1}{2^*}=1 \quad \mbox{ and }\quad 1<q<\frac{N+2}{N-2},
\end{equation}
then, by H\"older and Sobolev inequality one gets that  $T_u$ is linear and continuous. By Lax-Milgram theorem, there exists a unique $\phi_u\in {\mathcal D}^{1,2}(\R^N)$ such that
\begin{equation}\label{deft}
T_u(v)=\big( \phi_u,v\big)_{{\mathcal D}^{1,2}}\quad\mbox{ for all }v\in {\mathcal D}^{1,2}(\R^N).
\end{equation}
As a result, $\phi_u$ solves
$$
-\Delta \phi_u=K(x)|u|^q\quad \mbox{ in }\R^N,
$$
and
$$
\phi_u(x)=A_2 \intr \frac{K(y)|u|^p(y)}{|x-y|^{N-2}}dy \quad\mbox{  where $A_2$ corresponds to \eqref{rieszdef}.}
$$
Hence
\begin{equation}\label{pot}
\phi_u=I_2*(K|u|^q).
\end{equation}
More properties of $\phi_u$ are given in Lemma \ref{phiu} below. We should finally note that with $\phi_u$ given by \eqref{pot}, system \eqref{sp} reduces implicitly to the single equation
\begin{equation}\label{ssp}
-\Delta u+u+K(x) \phi_u |u|^{q-2}u=(I_\alpha*|u|^p)|u|^{p-2}u \quad\mbox{ in }\R^N.
\end{equation}

Let us remark that \eqref{ssp} has a variational structure. If $\frac{N+\alpha}{N}<p<\frac{N+\alpha}{N-2}$ and $q,r$ satisfy \eqref{k} then functional
$$
\mj(u)=\frac{1}{2}\intr(|\nabla u|^2+|u|^2)+\frac{1}{2q}\intr K(x)\phi_u |u|^q-\frac{1}{2p}\intr (\ia*|u|^p)|u|^p
$$
is well defined for all $u\in H^1(\R^N)$ and any critical point $u$ of $\mj$ is a weak solution to \eqref{ssp}.

Our existence result is the following.

\begin{theorem}\label{existence}
Assume $1<q<\frac{N+2}{N-2}$, $\frac{N+\alpha}{N}<p<\frac{N+\alpha}{N-2}$, $q<p$ and $K\in L^r(\R^N)$, with $r$ given by \eqref{k}.
Then, there exists $M>0$ such that for any $\|K\|_r<M$ problem \eqref{sp} has a solution $(u,\phi)\in H^1(\R^N)\times {\mathcal D}^{1,2}(\R^N)$. Moreover, $u$  is a ground state of \eqref{ssp}.
\end{theorem}
In order to deal with the lack of compactness of $H^1(\R^N)$ into the Lebesgue spaces $L^s(\R^N)$, $2\leq s\leq 2^*$, we rely on a careful analysis of the Palais-Smale (in short $(PS)$) sequences for $\mj$ restricted to its Nehari manifold ${\mathcal N}$. Roughly speaking, we have that any $(PS)$ sequence of $\mj\!\mid_{\mathcal N}$ either converges strongly to its weak limit or differs from it by a finite number of sequences, which are nothing but translated solutions of \eqref{gstate}, centered at points whose distances from the origin and whose interdistances go to infinity (see Proposition \ref{compact}). Then, a further evaluation of the energy levels of $\mj$ allows us to locate some ranges for which the compactness is still preserved. Such an approach was successfully applied for the Schr\"odinger-Poisson system \eqref{sp0} in \cite{CM2016,CV2010} and recently adapted to the study of the non-autonomous fractional Choquard equation in \cite{CL2016}. 
Unlike the approach in \cite{CL2016} where a direct energy estimation is possible due to the presence of suitable non-autonomous terms, we shall rely essentially on several nonlocal Brezis-Lieb type results as we describe in Section \ref{blieb}.

The remaining part of the paper is organised as follows. Section 3 contains some preliminary results which we will use in the study of the existence of a ground state to \eqref{sp}. Sections 4 and 5 contain the proofs of our main results.

\section{Preliminary results}

\subsection{Some properties of $\phi_u$}

\begin{lemma}\label{phiu} We have
\begin{enumerate}
\item[(i)] $\phi_u\geq 0$ for any $u\in H^1(\R^N)$;
\item[(ii)] $\phi_{tu}=t^q\phi_u$ for any $t>0$;
\item[(iii)] if $u_n\rightharpoonup u$ weakly in $H^1(\R^N)$, then $\phi_{u_n}\to \phi_u$ strongly in ${\mathcal D}^{1,2}(\R^N)$.
\end{enumerate}
\end{lemma}
\begin{proof} (i) and (ii) follow from the definition of $\phi_u$. 

(iii) For a proof of this part in dimension $N=3$ the reader may consult \cite[Proposition 2.2(a)]{CM2016}. Here we provide a different argument.

Let us note first that from the definition of $\phi_u$ in \eqref{deft} we deduce
$$
\|\phi_u\|_{{\mathcal D}^{1,2}}=\|T_u\|_{{\mathcal L}({\mathcal D}^{1,2})}.
$$
For any $v\in {\mathcal D}^{1,2}(\R^N)$ we have
$$
\begin{aligned}
|T_{u_n}(v)-T_u(v)|&\leq \intr K(x)\big||u_n|^q-|u|^q\big||v|\\
&\leq \|v\|_{{\mathcal D}^{1,2}}\Big(\intr K(x)^{\frac{2N}{N+2}}\big||u_n|^q-|u|^q\big|^{\frac{2N}{N+2}}  \Big)^{\frac{N+2}{2N}}.
\end{aligned}
$$
Using the continuous embedding of $H^1(\R^N)$ into $L^s(\R^N)$, $2\leq s\leq 2^*$ and Lemma \ref{bogachev} below, it follows that
$$
\big||u_n|^q-|u|^q\big|^{\frac{2N}{N+2}}\rightharpoonup 0\quad\mbox{ weakly in }L^{\frac{N+2}{q(N-2)}}(\R^N).
$$
Thus, since $K^{\frac{2N}{N+2}}\in L^{\frac{r(N+2)}{2N}}(\R^N)$ we deduce
$$
\begin{aligned}
\|\phi_{u_n}-\phi_u\|_{{\mathcal D}^{1,2}}&=\|T_{u_n}-T_u\|_{{\mathcal L}({\mathcal D}^{1,2})}\\
&\leq \Big(\intr K(x)^{\frac{2N}{N+2}}\big||u_n|^q-|u|^q\big|^{\frac{2N}{N+2}}  \Big)^{\frac{N+2}{2N}}\to 0.
\end{aligned}
$$
\end{proof}

\subsection{Some nonlocal versions of Brezis-Lieb lemma} \label{blieb}

In this part we collect some useful results in dealing with the existence of a ground state solution to \eqref{gstate}.

We first recall the concentration-compactness lemma of P.-L. Lions formulated in an inequality setting.
\begin{lemma}\label{cc}{\sf (\cite[Lemma I.1]{Lions1984},  \cite[Lemma 2.3]{MV2013})}

Let $s\in [2,2^*]$. There exists a constant $C>0$ such that for any $u\in H^1(\R^N)$ we have
$$
\int_{\R^N}|u|^s \leq C\|u\|\Big(\sup_{y\in \R^N} \int_{B_1(y)}|u|^s \Big)^{1-\frac{2}{s}}.
$$
\end{lemma}

\begin{lemma}\label{bogachev}{\sf (\cite[Proposition 4.7.12]{B2007})}

Let $s\in (1,\infty)$. Assume $(w_n)$ is a bounded sequence in $L^s(\R^N)$ that converges to $w$ almost everywhere. Then $w_n\rightharpoonup w$ weakly in $L^s(\R^N)$.
\end{lemma}
Using a similar proof to that in the original Brezis-Lieb lemma \cite[Theorem 2]{BL1983} (see also \cite[Proposition 4.7.30]{W2013}) we have

\begin{lemma}\label{blc}{\sf (Local Brezis-Lieb lemma)}

Let $s\in (1,\infty)$. Assume $(w_n)$ is a bounded sequence in $L^s(\R^N)$ that converges to $w$ almost everywhere. Then, for every $q\in [1,s]$ we have
$$
\lim_{n\to\infty}\int_{\R^N}\big| |w_n|^q-|w_n-w|^q-|w|^q\big|^{\frac{s}{q}}=0\,,
$$
and
$$
\lim_{n\to\infty}\int_{\R^N}\big| |w_n|^{q-1}w_n-|w_n-w|^{q-1}(w_n-w)-|w|^{q-1}w\big|^{\frac{s}{q}}=0.
$$
\end{lemma}

A first nonlocal version of Bezis-Lieb lemma in the literature appeared in \cite{MV2013} (see also \cite{MMV2015}) and reads as follows.

\begin{lemma}\label{nlocbl}{\sf (Nonlocal Brezis-Lieb lemma, \cite[Lemma 2.4]{MV2013})}

Let $\alpha\in (0,N)$ and $p\in [1,\frac{2N}{N+\alpha})$. Assume $(u_n)$ is a bounded sequence in $L^{\frac{2Np}{N+\alpha}}(\R^N)$ that converges almost everywhere to  some $u:\R^N\to \R$. Then
$$
\lim_{n\to \infty}\int_{\R^N}\Big|(I_\alpha*|u_n|^p)|u_n|^p-(I_\alpha*|u_n-u|^p)|u_n-u|^p-(I_\alpha*|u|^p)|u|^p  \Big|=0.
$$
\end{lemma}

Below we state and prove another nonlocal version of Brezis-Lieb lemma.

\begin{lemma}\label{anbl}

Let $\alpha\in (0,N)$ and $p\in [1,\frac{2N}{N+\alpha})$. Assume $(u_n)$ is a bounded sequence in $L^{\frac{2Np}{N+2}}(\R^N)$ that converges almost everywhere to $u$. Then, for any $h\in L^{\frac{2Np}{N+\alpha}}(\R^N)$ we have
$$
\lim_{n\to \infty}\int_{\R^N}(I_\alpha*|u_n|^p)|u_n|^{p-2}u_nh=\int_{\R^N} (I_\alpha*|u|^p)|u|^{p-2}uh.
$$
\end{lemma}
\begin{proof} Using $h=h^+-h^-$, it is enough to prove our lemma for $h\geq 0$. 
Denote $v_n=u_n-u$ and observe that
\begin{equation}\label{bl1}
\begin{aligned}
\intr (\ia*|u_n|^p)|u_n|^{p-2}u_nh=& \intr [\ia*(|u_n|^p-|v_n|^p)](|u_n|^{p-2}u_nh-|v_n|^{p-2}v_nh)\\
&+\intr [\ia*(|u_n|^p-|v_n|^p)]|v_n|^{p-2}v_nh\\
&+\intr [\ia*(|u_n|^{p-2}u_nh-|v_n|^{p-2}v_n h)|v_n|^p\\
&+\intr (\ia*|v_n|^p)|v_n|^{p-2}v_nh.
\end{aligned}
\end{equation}
Apply Lemma \ref{blc} with $q=p$, $s=\frac{2Np}{N+\alpha}$ by taking respectively $(w_n,w)=(u_n,u)$ and then $(w_n,w)=(u_nh^{1/p}, u h^{1/p})$. We find
$$
\left\{
\begin{aligned}
&|u_n|^p-|v_n|^p\to |u|^p \\
&|u_n|^{p-2}u_nh-|v_n|^{p-2}v_nh\to |u|^{p-2}uh
\end{aligned}
\right.
\quad\mbox{ strongly in }\; L^{\frac{2N}{N+\alpha}}(\R^N).
$$
Using now the Hardy-Littlewood-Sobolev inequality \eqref{hls1} we obtain
\begin{equation}\label{est00}
\left\{
\begin{aligned}
&\ia*(|u_n|^p-|v_n|^p)\to \ia*|u|^p \\
&\ia*(|u_n|^{p-2}u_nh-|v_n|^{p-2}v_nh)\to \ia*(|u|^{p-2}uh)
\end{aligned}
\right.
\quad\mbox{ strongly in }\; L^{\frac{2N}{N-\alpha}}(\R^N).
\end{equation}
Also, by Lemma \ref{bogachev} we have
\begin{equation}\label{est01}
|u_n|^{p-2}u_n h\rightharpoonup |u|^{p-2}uh,\; |v_n|^p\rightharpoonup 0,\; |v_n|^{p-2}v_nh\rightharpoonup 0\quad \mbox{ weakly in }\; L^{\frac{2N}{N+\alpha}}(\R^N).
\end{equation}
Combining \eqref{est00}-\eqref{est01} we find
\begin{equation}\label{est02}
\left\{
\begin{aligned}
&\lim_{n\to \infty} \intr [\ia*(|u_n|^p-|v_n|^p)](|u_n|^{p-2}u_nh-|v_n|^{p-2}v_n h)=\int_{\R^N} (I_\alpha*|u|^p)|u|^{p-2}uh,\\
&\lim_{n\to \infty} \intr [\ia*(|u_n|^p-|v_n|^p)]|v_n|^{p-2}v_nh=0,\\
&\lim_{n\to \infty} \intr [\ia*(|u_n|^{p-2}u_nh-|v_n|^{p-2}v_nh)|v_n|^p=0.
\end{aligned}
\right.
\end{equation}
By H\"older's inequality and Hardy-Littlewood-Sobolev inequality \eqref{hls2} with $s=t=\frac{2N}{N+\alpha}$  we have
\begin{equation}\label{est03}
\begin{aligned}
\left| \intr (\ia*|v_n|^{p}))|v_n|^{p-2}v_nh \right|
&\leq \|v_n\|^p_{\frac{2Np}{N+\alpha}}\||v_n|^{p-1}h\|_{\frac{2N}{N+\alpha}}\\
&\leq C \||v_n|^{p-1}h\|_{\frac{2N}{N+\alpha}}.
\end{aligned}
\end{equation}
On the other hand, by Lemma \ref{bogachev} we have $v_n^{\frac{2N(p-1)}{N+\alpha}}\rightharpoonup 0$ weakly in 
$L^{\frac{p}{p-1}}(\R^N)$ so 
$$
\||v_n|^{p-1}h\|_{\frac{2N}{N+\alpha}}=\left(\intr |v_n|^{\frac{2N(p-1)}{N+\alpha}}|h|^{\frac{2N}{N+\alpha}}  \right)^{\frac{N+\alpha}{2N}}\to 0.
$$
Thus, from \eqref{est03}  have
\begin{equation}\label{est04}
\lim_{n\to \infty} \intr (\ia*|v_n|^{p}))|v_n|^{p-2}v_nh=0.
\end{equation}
Passing to the limit in \eqref{bl1}, from \eqref{est02} and \eqref{est04} we reach the conclusion.
\end{proof}

\section{Proof of Theorem \ref{nonex}}

\subsection{A Pohozaev identity}

The main tool in proving Theorem \ref{nonex} is the following Pohozaev type identity.

\begin{proposition}\label{poh}
Let $(u,\phi)$ be a solution of \eqref{sp} that satisfies \eqref{reg1}-\eqref{reg2}.  Then
$$
\intr \!\Big(\frac{N-2}{2}|\nabla u|^2+\frac{N}{2}|u|^2\Big)+\frac{N+2}{2q}\intr K(x)\phi |u|^q+\frac{1}{q}\intr \phi |u|^q x\cdot \nabla K(x)=\frac{N+\alpha}{2p}\intr(\ia*|u|^p)|u|^p.
$$
\end{proposition}
\begin{proof} Let $\varphi\in C^1_c(\R^N)$ be such that $\varphi\equiv 1$ on $B_1(0)$. For $\lambda>0$ set 
$$
v_\lambda(x)=\varphi(\lambda x)x\cdot \nabla u.
$$
Then $v_\lambda\in W^{1,2}(\R^N)\cap L^{\frac{2Np}{N+\alpha}}(\R^N)$ and from the first equation of \eqref{sp} we have
\begin{equation}\label{test}
\intr \nabla u\cdot \nabla v_\lambda+\intr uv_\lambda+\intr K(x)\phi |u|^{q-2}uv_\lambda=\intr (\ia*|u|^p)|u|^{p-2}uv_\lambda.
\end{equation}
Let us next analyse term by term the above equation.
Since $u\in W^{2,2}_{loc}(\R^N)$ we have
$$
\begin{aligned}
\intr \nabla u\cdot \nabla v_\lambda&=\intr \varphi(\lambda x)\Big[|\nabla u|^2+x\cdot \nabla\Big(\frac{|\nabla u|^2}{2}\Big)(x)\Big]dx\\
&=-\intr \big[\lambda x\cdot \nabla \varphi(\lambda x)+(N-2) \varphi(\lambda x)\big] \frac{|\nabla u(x)|^2}{2}dx.
\end{aligned}
$$
By Lebesgue dominated convergence theorem, we find
$$
\lim_{\lambda\to 0} \intr \nabla u\cdot \nabla v_\lambda=-\frac{N-2}{2}\intr|\nabla u|^2.
$$
Next,
$$
\begin{aligned}
\intr  u v_\lambda &=\intr u(x) \varphi(\lambda x) x\cdot \nabla u(x)dx\\
&=\intr \varphi(\lambda x)x\cdot \nabla\Big( \frac{|u|^2}{2} \Big)(x)dx\\
&=-\intr \big[\lambda x\cdot \nabla \varphi(\lambda x)+N \varphi(\lambda x)\big] \frac{| u(x)|^2}{2}dx.
\end{aligned}
$$
Again by Lebesgue dominated convergence theorem we deduce
$$
\lim_{\lambda\to 0} \intr u v_\lambda=-\frac{N}{2}\intr| u|^2.
$$
Further we have
$$
\begin{aligned}
\intr (\ia& *|u|^p)|u|^{p-2}uv_\lambda \\
=&\frac{1}{p}\intr\intr \ia(x-y)|u|^p(y) \varphi(\lambda x) x\cdot \nabla(|u|^p)(x)dxdy\\
=&\frac{1}{2p}\intr\intr \ia(x-y)\Big\{ |u|^p(y) \varphi(\lambda x) x\cdot \nabla(|u|^p)(x)+
|u|^p(x) \varphi(\lambda y) y\cdot \nabla(|u|^p)(y)\Big\}dxdy\\
=&-\frac{1}{p}\intr\intr \ia(x-y)|u|^p(x)|u|^p(y)\Big[N\varphi(\lambda x)+\lambda x\cdot \nabla \varphi(\lambda x)\Big] dxdy\\
&+\frac{N-\alpha}{2p}\intr\intr \frac{(x-y)(x\varphi(\lambda x)-y\varphi(\lambda y))}{|x-y|^2}\ia(x-y) |u|^p(x)|u|^p(y)dx dy,
\end{aligned}
$$
which yields
$$
\lim_{\lambda\to 0} \intr (\ia *|u|^p)|u|^{p-2}uv_\lambda =-\frac{N+\alpha}{2p}\intr(\ia*|u|^p)|u|^p.
$$
Note that the regularity of $K,u$, $\phi$ and the second equation of \eqref{sp} allow us to derive
$$
\phi=I_2*(K|u|^q).
$$ 
Thus, we have
$$
\begin{aligned}
\intr & K(x)\phi |u|^{q-2}uv_\lambda\\
=& \frac{1}{q}\intr (I_2*K|u|^q)K(x) \varphi(\lambda x) x\cdot\nabla(|u|^q) dx\\
=&\frac{1}{q}\intr\intr I_2(x-y)K(x)K(y)|u|^q(y) \varphi(\lambda x) x\cdot \nabla(|u|^q)(x)dxdy\\
=&\frac{1}{2q}\intr\intr I_2(x-y)K(x)K(y)\Big\{ |u|^q(y) \varphi(\lambda x) x\cdot \nabla(|u|^q)(x)+
|u|^q(x) \varphi(\lambda y) y\cdot \nabla(|u|^q)(y)\Big\}dxdy\\
=&-\frac{1}{q}\intr\intr I_2(x-y)K(y) |u|^q(x) |u|^q(y)\Big[NK(x) \varphi(\lambda x)+K(x) \lambda x\cdot \nabla \varphi(\lambda x)+\varphi(\lambda x) x\cdot \nabla K(x)\Big] dxdy\\
&+\frac{N-2}{2q}\intr\intr \frac{(x-y)(x\varphi(\lambda x)-y\varphi(\lambda y))}{|x-y|^2} I_2(x-y)K(x)K(y) |u|^p(x)|u|^p(y)dx dy\\
=&-\frac{1}{q}\intr \phi(x) |u|^q(x) \Big[NK(x) \varphi(\lambda x)+K(x) \lambda x\cdot \nabla \varphi(\lambda x)+\varphi(\lambda x) x\cdot \nabla K(x)\Big] dx\\
&+\frac{N-2}{2q}\intr\intr \frac{(x-y)(x\varphi(\lambda x)-y\varphi(\lambda y))}{|x-y|^2} I_2(x-y)K(x)K(y) |u|^p(x)|u|^p(y)dx dy.
\end{aligned}
$$
We obtain
$$
\lim_{\lambda\to 0}\intr  K(x)\phi |u|^{q-2}uv_\lambda=-\frac{N+2}{2q}\intr K(x)\phi |u|^q-\frac{1}{q}\intr \phi |u|^q x\cdot \nabla K(x).
$$
Passing now to the limit in \eqref{test} we obtain the conclusion.
\end{proof}

\subsection{Proof of Theorem \ref{nonex} completed}

Let $(u,\phi)$ be a solution of \eqref{sp} which satisfies \eqref{reg1}-\eqref{reg2}. It is enough to show that $u\equiv 0$ as the second equation  of \eqref{sp} together with $\phi\in H^1(\R^N)$ will imply $\phi\equiv 0$. Suppose by contradiction that the solution $(u,\phi)$ satisfies $u\not\equiv 0$.

For convenience, let us denote
\begin{equation}\label{notation}
A(u)=\intr K(x)\phi |u|^q\,,\;\; B(u)=\intr\phi |u|^q \; x\cdot \nabla K(x)\,,\quad C(u)=\intr (\ia*|u|^p)|u|^p.
\end{equation}
From Proposition \ref{poh} we have
\begin{equation}\label{poh1}
\frac{N-2}{2}\|\nabla u\|_2^2+\frac{N}{2}\|u\|_2^2+\frac{N+2}{2q}A(u)+\frac{1}{q}B(u)=\frac{N+\alpha}{2p}C(u).
\end{equation}
Since $u$ is a solution of \eqref{ssp} we also have
\begin{equation}\label{poh2}
C(u)=\|u\|^2+A(u).
\end{equation}

(i) Assume $x\cdot\nabla K(x)+\gamma K(x)\geq 0$ in $\R^N$ for some $\gamma\in (-\infty,\frac{N+2}{2})$ and that \eqref{nex1} holds. Then
$$
B(u)\geq -\gamma A(u)
$$
so that from \eqref{poh1} and \eqref{poh2} we obtain
$$
\frac{N+\alpha}{2p}\Big(\|u\|^2+A(u)\Big)>\frac{N-2}{2}\|u\|^2+\frac{N+2-2\gamma}{2q}A(u)
$$
that is,
$$
\frac{N+\alpha-p(N-2)}{p}\|u\|^2>\Big( \frac{N+2-2\gamma}{q}-\frac{N+\alpha}{p}\Big)A(u).
$$
But this last inequality is impossible since $\|u\|> 0$, $A(u)\geq 0$ and $p,q,N,\alpha,\gamma$ satisfy \eqref{nex1}.

(ii) Assume $x\cdot\nabla K(x)+\gamma K(x)\leq 0$ in $\R^N$ for some $\gamma\in \R$ and that \eqref{nex2} holds. It follows that
$$
B(u)\leq -\gamma A(u)
$$
so that \eqref{poh1} together with \eqref{poh2} yield
$$
\frac{N+\alpha}{2p}\Big(\|u\|^2+A(u)\Big)<\frac{N}{2}\|u\|^2+\frac{N+2-2\gamma}{2q}A(u)
$$
that is,
$$
\frac{N+\alpha-pN}{p}\|u\|^2<\Big( \frac{N+2-2\gamma}{q}-\frac{N+\alpha}{p}\Big)A(u).
$$
Note that the above inequality is impossible since $\|u\|>0$, $A(u)\geq 0$ and $p,q,N,\alpha,\gamma$ satisfy \eqref{nex2}.
This concludes our proof.

\section{Proof of Theorem \ref{existence}}

\subsection{The Nehari manifold associated with \eqref{ssp}}
Define the Nehari manifold associated with $\mj$  as
\begin{equation}\label{nm}
{\mathcal N}=\{u\in H^1(\R^N)\setminus\{0\}: \langle \mj'(u),u\rangle=0\}
\end{equation}
and let 
$$
m_{\mj}=\inf_{u\in {\mathcal N}}\mj(u).
$$

Remark that for $u\in H^1(\R^N)\setminus\{0\}$ and $t>0$ we have
$$
\langle \mj'(tu),tu\rangle=t^2\|u\|^2+t^{2q}\intr K(x)\phi_u |u|^{q-2}u-t^{2p}\intr (\ia*|u|^p)|u|^p.
$$
Since $p>q>1$, the equation $\langle \mj'(tu),tu\rangle=0$ has a unique positive solution $t=t(u)$ and the corresponding element $t(u)u\in {\mathcal N}$ is called the {\it projection of $u$} on ${\mathcal N}$. The main properties of the Nehari manifold which we use in this paper are stated below.

\begin{proposition}\label{nehari}
\begin{enumerate}
\item[(i)] $\mj\!\mid_{\mathcal N}$ is bounded from below by a positive constant;
\item[(ii)] If $u$ is a critial point of $\mj$ in ${\mathcal N}$ then $u$ is a free critical point of $\mj$;
\end{enumerate}
\end{proposition}
\begin{proof} (i) Using the Hardy-Littlewood-Sobolev inequality \eqref{hls2} together with the continuous embedding $H^1(\R^N) 	\hookrightarrow L^{\frac{2Np}{N+\alpha}}(\R^N)$, for any $u\in {\mathcal N}$ we have
$$
\begin{aligned}
0=\langle \mj'(u),u\rangle& =\|u\|^2+\intr K(x)\phi_u |u|^{q}-\intr (\ia*|u|^p)|u|^p\\
&\geq \|u\|^2-C\|u\|^{2p}.
\end{aligned}
$$
Hence, there exists $C_0>0$ such that
\begin{equation}\label{cnot}
\|u\|\geq C_0>0\quad\mbox{for all }u\in {\mathcal N}.
\end{equation}
Using this fact we have
$$
\begin{aligned}
\mj(u)&=\Big(\frac{1}{2}-\frac{1}{2p}\Big)\|u\|^2+\Big(\frac{1}{2q}-\frac{1}{2p}\Big)\intr K(x)\phi_u |u|^q\\
&\geq \Big(\frac{1}{2}-\frac{1}{2p}\Big) C_0^2>0.
\end{aligned}
$$

(ii) For $u\in H^1(\R^N)$ let ${\mathcal G}(u)=\langle \mj'(u),u\rangle$. If $u\in {\mathcal N}$, by \eqref{cnot} we obtain
\begin{equation}\label{cnot1}
\begin{aligned}
\langle {\mathcal G}'(u),u\rangle&=2\|u\|^2+2q\intr K(x)\phi_u |u|^{q}-2p\intr (\ia*|u|^p)|u|^p\\
&=2(1-q)\|u\|^2-2(p-q)\intr (\ia*|u|^p)|u|^p\\
&\leq -2(q-1)\|u\|^2\\
&<-2(q-1)C_0.
\end{aligned}
\end{equation}
Assume now that $u\in {\mathcal N}$ is a critical point of $\mj$ in ${\mathcal N}$.  By the Lagrange multiplier theorem, there exists $\lambda\in \R$ such that
$\mj'(u)=\lambda {\mathcal G}'(u)$. 
In particular $\langle \mj '(u),u\rangle=\lambda \langle {\mathcal G}'(u),u\rangle$. Since $\langle {\mathcal G}'(u),u\rangle<0$, it follows that $\lambda=0$ so $\mj '(u)=0$.

\end{proof}

\subsection{A compactness result}\label{compc}
Let
$$
{\mathcal E}:H^1(\R^N)\to \R,\quad {\mathcal E}(u)=\frac{1}{2}\intr(|\nabla u|^2+|u|^2)-\frac{1}{2p}\intr (\ia*|u|^p)|u|^p,
$$
be the energy functional corresponding to \eqref{gstate}. Also, consider its Nehari manifold
$$
{\mathcal N}_{\mathcal E}=\{u\in H^1(\R^N)\setminus\{0\}: \langle \me'(u),u\rangle=0\}
$$
and let 
$$
m_{\me}=\inf_{u\in {\mathcal N}_{\mathcal E}}\me(u).
$$
\begin{proposition}\label{compact}
Let $(u_n)\subset{\mathcal N}$ be a $(PS)$ sequence of $\mj\!\mid_{\mathcal N}$, that is,
\begin{enumerate}
\item[(a)] $(\mj(u_n))$ is bounded;
\item[(b)] $\Big(\mj\!\mid_{\mathcal N}\Big)'(u_n)\to 0$ strongly in $H^{-1}(\R^N)$.
\end{enumerate}
Then, there exists a solution $u\in H^1(\R^N)$ of \eqref{ssp} such that replacing $(u_n)$ with a subsequence the following alternative holds
\smallskip

\noindent (1) either $u_n\to u$ strongly in $H^1(\R^N)$;

or
\smallskip

\noindent (2) $u_n\rightharpoonup u$ weakly (but not strongly) in $H^1(\R^N)$ and there exists 
a positive integer $k\geq 1$, $k$ functions $u_1,u_2,\dots, u_k\in H^1(\R^N)$ which are nontrivial  weak solutions to \eqref{gstate} and $k$ sequence of points $(y_{n,1})$, $(y_{n,2})$, $\dots$, $(y_{n,k})\subset \R^N$
such that:
\begin{enumerate}
\item[(i)] $|y_{n,j}|\to \infty$ and $|y_{n,j}-y_{n,i}|\to \infty$  if $i\neq j$, $n\to \infty$;
\item[(ii)] $\di u_n-\sum_{j=1}^ku_j(\cdot+y_{n,j})\to u$ in $H^1(\R^N)$;
\item[(iii)] $\di \mj(u_n)\to \mj(u)+\sum_{j=1}^k \me(u_j)$;
\end{enumerate}
\end{proposition}
\begin{proof}
Because $(u_n)$ is bounded in $H^1(\R^N)$, there exists $u\in H^1(\R^N)$ such that, up to a subsequence, we have
\begin{equation}\label{firstconv}
\left\{
\begin{aligned}
u_n& \rightharpoonup u \quad\mbox{ weakly in }H^1(\R^N),\\
u_n &\rightharpoonup u\quad\mbox{ weakly in }L^s(\R^N),\; 2\leq s\leq 2^*,\\
u_n & \to u\quad\mbox{ a.e. in }\R^N.
\end{aligned}
\right.
\end{equation}
We also need the following result:

\begin{lemma}\label{est1}
We have
\begin{enumerate}
\item[(i)] $\di \intr K(x) \phi_{u_n} |u_n|^q=\intr K(x) \phi_{u} |u|^q+o(1)$;
\item[(ii)] $\di \intr K(x) \phi_{u_n} |u_n|^{q-2}u_nh=\intr K(x) \phi_{u} |u|^{q-2}uh+o(1),\;$ for all $h\in H^1(\R^N)$.
\end{enumerate}
\end{lemma}
\begin{proof} We shall prove only (ii) as the (i) part is similar. 

Note first that
\begin{equation}\label{est3}
\begin{aligned}
\left| \intr \!K(x) \phi_{u_n} |u_n|^{q-2}u_n h-\intr\! K(x) \phi_{u} |u|^{q-2}uh\right|\leq & \intr \! |K(x)||\phi_{u_n}-\phi_u||u_n^{q-1}h|\\
&+ \left| \intr \!K(x) \phi_{u}h(|u_n|^{q-2}u_n-|u|^{q-2}u)\right|.
\end{aligned}
\end{equation}
Using Lemma \ref{phiu}(iii) and H\"older's inequality we find
\begin{equation}\label{est4}
\begin{aligned}
\intr\! |K(x)||\phi_{u_n}-\phi_u||u_n^{q-1}h|& \leq  \|K\|_r\|\phi_{u_n}-\phi_u\|_{2^*}\|u_n^{q-1}\|_{\frac{2^*}{q-1}}\|h\|_{2^*}\\
&=\|K\|_r\|\phi_{u_n}-\phi_u\|_{2^*}\|u_n\|_{2^*}\|h\|_{2^*}=o(1).
\end{aligned}
\end{equation}
By Lemma \ref{bogachev} we have $|u_n|^{q-2}u_n\rightharpoonup |u|^{q-2}u$ weakly in $L^{\frac{2^*}{q-1}}(\R^N)$. 

Since $K(x)\phi_uh\in L^{\frac{2^*}{2^*-(q-1)}}(\R^N)$ it follows that
\begin{equation}\label{est5}
\intr \!K(x) \phi_{u}h(|u_n|^{q-2}u_n-|u|^{q-2}u)=o(1).
\end{equation}
Now, the proof follows by combining \eqref{est3}-\eqref{est5}.
\end{proof}

We now return to the proof of Proposition \ref{compact}. By \eqref{firstconv}, Lemma \ref{anbl} and Lemma \ref{est1}(ii) it follows that $\mj'(u)=0$ so $u\in H^1(\R^N)$ is a solution of \eqref{ssp}.

If $u_n\to u$ strongly in $H^1(\R^N)$ then the first alternative in the statement of Proposition \ref{compact} holds and we are done. Assume in the following that $(u_n)$ does not converge strongly in $H^1(\R^N)$ to $u$ and define
$z_{n,1}=u_n-u$. Then $(z_{n,1})$ converges weakly and not strongly to zero in $H^1(\R^N)$ and
\begin{equation}\label{bl2}
\|u_n\|^2=\|u\|^2+\|z_{n,1}\|^2+o(1).
\end{equation}
By Lemma \ref{nlocbl} we have
\begin{equation}\label{bl3}
\int_{\R^N} (I_\alpha*|u_n|^p)|u_n|^p=\intr (I_\alpha*|u|^p)|u|^p+\intr (I_\alpha*|z_{n,1}|^p)|z_{n,1}|^p+o(1). 
\end{equation}
 Using \eqref{bl2}, \eqref{bl3} and Lemma \ref{est1}(i) we deduce
\begin{equation}\label{est6}
{\mathcal J}(u_n)=\mj(u)+{\mathcal E}(z_{n,1})+o(1).
\end{equation}
For any $h\in H^1(\R^N)$, by Lemma \ref{anbl} and Lemma \ref{est1}(ii) we have
\begin{equation}\label{est7}
\langle \me '(z_{n,1}), h\rangle=o(1).
\end{equation}
Next, by Lemma \ref{nlocbl} and Lemma \ref{est1}(i) we have
$$
\begin{aligned}
0=\langle \mj '(u_n), u_n \rangle&=\langle \mj '(u),u\rangle+\langle \me'(z_{n,1}), z_{n,1} \rangle+o(1)\\
&=\langle \me'(z_{n,1}), z_{n,1}\rangle+o(1),
\end{aligned}
$$
which yields
\begin{equation}\label{est8}
\langle \me '(z_{n,1}), z_{n,1}\rangle=o(1).
\end{equation}
Let
$$
\delta:=\limsup_{n\to \infty}\Big(\sup_{y\in \R^N} \int_{B_1(y)}|z_{n,1}|^{\frac{2Np}{N+\alpha}}\Big)\geq 0.
$$
We claim that $\delta>0$.
Indeed, if $\delta=0$, by Lemma \ref{cc} we deduce $z_{n,1}\to 0$ strongly in $L^{\frac{2Np}{N+\alpha}}(\R^N)$. Then, by Hardy-Littlewood-Sobolev inequality \eqref{hls2} we find
$$
\intr (I_\alpha*|z_{n,1}|^p)|z_{n,1}|^p=o(1).
$$
This fact combined with \eqref{est8} yields $z_{n,1}\to 0$ strongly in $H^1(\R^N)$ in contradiction to our assumption.

Hence, $\delta>0$ so that we may find $y_{n,1}\in \R^N$ with
\begin{equation}\label{est9}
\int_{B_1(y_{n,1})}|z_{n,1}|^{\frac{2Np}{N+\alpha}}>\frac{\delta}{2}.
\end{equation}
Considering the sequence $(z_{n,1}(\cdot+y_{n,1}))$, there exists $u_1\in H^1(\R^N)$ such that, up to a subsequence, we have  
$$
\begin{aligned}
z_{n,1}(\cdot+y_{n,1})&\rightharpoonup u_1\quad\mbox{ weakly in } H^1(\R^N),\\
z_{n,1}(\cdot+y_{n,1})&\to u_1\quad\mbox{ strongly in } L_{loc}^{\frac{2Np}{N+\alpha}}(\R^N),\\
z_{n,1}(\cdot+y_{n,1})&\to u_1\quad\mbox{ a.e. in } \R^N.
\end{aligned}
$$
Passing to the limit in \eqref{est9} we find
$$
\int_{B_1(0)}|u_{1}|^{\frac{2Np}{N+\alpha}}\geq \frac{\delta}{2},
$$
so $u_1\not\equiv 0$. Also, since $(z_{n,1})$ converges weakly to zero in $H^1(\R^N)$  it follows that $(y_{n,1})$ is unbounded. Passing to a subsequence we may assume $|y_{n,1}|\to \infty$. From \eqref{est8} we also obtain $\me'(u_1)=0$, so $u_1$ is a nontrivial solution of \eqref{gstate}.

Set next
$$
z_{n,2}(x)=z_{n,1}(x)-u_1(x-y_{n,1}).
$$
As above we have
$$
\|z_{n,1}\|^2=\|u_1\|^2+\|z_{n,2}\|^2+o(1).
$$
and by Lemma \ref{nlocbl} we derive
$$
\int_{\R^N} (I_\alpha*|z_{n,1}|^p)|z_{n,1}|^p=\intr (I_\alpha*|u_1|^p)|u_1|^p+\intr (I_\alpha*|z_{n,2}|^p)|z_{n,2}|^p+o(1). 
$$
Thus,
$$
\me (z_{n,1})=\me(u_1)+\me(z_{n,2})+o(1)
$$
so, by \eqref{est6} one has
$$
\mj (u_n)=\mj (u)+\me(u_1)+\me(z_{n,2})+o(1).
$$
Using the above techniques one can also derive
$$
\langle \me'(z_{n,2}),h\rangle =o(1)\quad\mbox{ for any }h\in H^1(\R^N)
$$
and
$$
\langle \me'(z_{n,2}),z_{n,2}\rangle =o(1).
$$
If $(z_{n,2})$ converges strongly to zero, the proof finishes (and take $k=1$ in the statement of Proposition \ref{compact}). Assuming that $z_{n,2}\rightharpoonup 0$ weakly and not strongly in $H^1(\R^N)$, we iterate the process. In $k$ number of steps we find a set of sequences $(y_{n,j})\subset \R^N$, $1\leq j\leq k$ with 
$$
|y_{n,j}|\to \infty\quad\mbox{  and }\quad |y_{n,i}-y_{n,j}|\to \infty\quad\mbox{  as }\; i\neq j, n\to \infty
$$
and $k$ nontrivial solutions  $u_1$, $u_2$, $\dots$, $u_k\in H^1(\R^N)$ of \eqref{gstate} such that, denoting 
$$
z_{n,j}(x):=z_{n,j-1}(x)-u_{j-1}(x-y_{n,j-1})\,, \quad 2\leq j\leq k,
$$ 
we have
$$
z_{n,j}(x+y_{n,j})\rightharpoonup u_j\quad\mbox{weakly in }\; H^1(\R^N)
$$
and
$$
\mj(u_n)= \mj(u)+\sum_{j=1}^k \me(u_j)+\me(z_{n,k})+o(1).
$$
Since  $\me(u_j)\geq m_{\me}$ and $(\mj(u_n))$ is bounded, the process can be iterated only a finite number of times. This concludes our proof.
\end{proof}

\begin{cor}\label{corr1} 
Let $c\in (0,m_\me)$. Then, any $(PS)_c$ sequence of $\mj\!\mid_{\mathcal N}$ is relatively compact. 
\end{cor}
\begin{proof}
Let $(u_n)$ be a $(PS)_c$ sequence of $\mj\!\mid_{\mathcal N}$. Since $\me(u_j)\geq m_\me$ in Proposition \ref{compact}, it follows that up to a subsequence $u_n\to u$ strongly in $H^1(\R^N)$ and $u$ is a solution of \eqref{ssp}. 
\end{proof}

\subsection{Proof of Theorem \ref{existence} completed}

The proof of Theorem \ref{existence} relies essentially on the following result. 

\begin{lemma}\label{rest}
There exists $M>0$ such that if $K\in L^r(\R^N)$ and $\|K\|_r<M$ then
$$
m_{\mj}<m_\me.
$$
\end{lemma}
\begin{proof}   
Denote by $w\in H^1(\R^N)$ the ground state solution of \eqref{gstate}. By \cite[Theorem 1]{MV2013} we know that such a ground state exists. Let $tw$ be the projection of $w$ on ${\mathcal N}$, that is, 
$t=t(w)>0$ is the unique real number such that $tw\in {\mathcal N}$ (with ${\mathcal N}$ defined in \eqref{nm}).  Denote
$$
A(w)=\intr K(x) \phi_w |w|^q\,,\quad B(w)=\intr (\ia*|w|^p)|w|^p.
$$
Since $w\in {\mathcal N}_{\mathcal E}$ and $tw\in {\mathcal N}$ we have
\begin{equation}\label{g1}
\|w\|^2=B(w)
\end{equation}
and
$$
t^2\|w\|^2+t^{2q} A(w)=t^{2p}B(w).
$$
From the above equalities we find $t>1$. Further, by H\"older inequality we have
\begin{equation}\label{g2}
A(w)\leq \|K\|_r\|\phi_w\|_{2^*}\|w\|^q_{2^*}.
\end{equation}
From \eqref{g1} and \eqref{g2} we deduce
$$
\begin{aligned}
m_{\mj}&\leq \mj(tw)=\frac{1}{2}t^2\|w\|^2+\frac{1}{2q}t^{2q} A(w)-\frac{1}{2p}t^{2p}B(w)\\
&=\Big(\frac{t^2}{2}-\frac{t^{2p}}{2p}\Big)\|w\|^2+\frac{t^{2q}}{2q} \|K\|_r\|\phi_w\|_{2^*}\|w\|^q_{2^*}.
\end{aligned}
$$
Since $t>1$, by letting $\|K\|_r$ small, it follows that
$$
m_{\mj}<\Big(\frac{1}{2}-\frac{1}{2p}\Big)\|w\|^2=\me(w)=m_{\me}.
$$
\end{proof}

By Ekeland Variational Principle, for any $n\geq 1$ there exists $u_n\in {\mathcal N}$ such that
$$
\begin{aligned}
\mj(u_n)&\leq m_{\mj}+\frac{1}{n}&&\quad\mbox{ for all }n\geq 1,\\
\mj(u_n)&\leq \mj(v)+\frac{1}{n}\|v-u_n\| &&\quad\mbox{ for all } v\in {\mathcal N}, n\geq 1.
\end{aligned}
$$
From here we easily deduce that $(u_n)\subset {\mathcal N}$ is a $(PS)_{m_{\mj}}$ sequence for $\mj\!\mid_{\mathcal N}$. Using Lemma \ref{rest} and Corollary \ref{corr1} it follows that up to a subsequence, $(u_n)$ converges strongly to some u$\in H^1(\R^N)$ which is a ground state solution of $\mj$.

\end{document}